\documentclass[10pt,a4paper]{article}
\usepackage{amsmath}
\usepackage{amsfonts}
\usepackage{amssymb}
\usepackage{amsthm}
\usepackage{graphicx}
\usepackage{hyperref}

\usepackage{algorithm}
\usepackage{algpseudocode}

\newtheorem{theorem}{Theorem}[section]

\newtheorem{definition}[theorem]{Definition}
\newtheorem{example}[theorem]{Example}
\newtheorem{remark}[theorem]{Remark}

\newcommand{\sgn}{\mathop{}\!\mathrm{sgn}}
\newcommand{\dist}{\mathop{}\!\mathrm{d}}

\begin{document}
%%%%%%%%%%%%%%%%

\title{Asymptotic KKT Conditions for Continuous-Time Nonlinear Programming}

\author{
Mois\'es R. C. do Monte\footnote{moises.monte@ufu.br} \\
Pontal Institute of Exact and Natural Sciences, \\
Federal University of Uberl\^andia (UFU),  \\ 
Ituiutaba, Minas Gerais, Brazil 
\and
Rodrigo B. Moreira\footnote{rbmoreira@uesc.br} \\
Department of Exact Sciences, \\
State University of Santa Cruz (UESC), \\
Ilhéus, Bahia, Brazil
\and
Valeriano A. de Oliveira\footnote{valeriano.oliveira@unesp.br} \\
Department of Mathematics, \\
S\~ao Paulo State University (UNESP), \\
%Institute of Biosciences, Humanities and Exact Sciences, \\
S\~ao Jos\'e do Rio Preto, S\~ao Paulo, Brazil
}

\date{}

\maketitle

\begin{abstract}
This paper addresses the class of continuous-time nonlinear programming problems with equality and inequality constraints. The paper presents necessary optimality conditions of the sequential form. To be more precise, a sequence of solutions converging to the optimal solution is demonstrated to exist, and such that Karush-Kuhn-Tucker-type conditions are satisfied asymptotically. It is shown that these sequential Karush-Kuhn-Tucker-type conditions also become sufficient for optimality under convexity assumptions. Sequential optimality conditions are a valuable tool for determining when to terminate a numerical method of solution. In this regard, an augmented Lagrangian-type method is proposed for numerically solving continuous-time programming problems. A convergence analysis concerning viability and optimality is presented. The performance of the method is evaluated by applying it to solve instances of continuous-time problems found in the literature.

\textbf{Keywords:} Continuous-time programming, Asymptotic optimality conditions, KKT, Augmented Lagrangian method.
\end{abstract}

%%%%%%%%%%%%%%%%%%%%%%%%%%%%%%%%%%%%%%%%%%%%%%%%%%%%%%%%%%%%%%%%%%%%%%

\section{Introduction} \label{sec:Intro}

The concept of continuous-time programming is frequently referenced in the literature and was initially proposed by Bellman \cite{bellman:1953} in his investigations of specific dynamic models of production and inventory, termed ``bottleneck processes''. In his initial formulation, Bellman considered a linear programming problem. However, his approach subsequently underwent significant expansion to encompass more general forms of continuous-time linear programming problems and certain classes of continuous-time nonlinear programming problems. For a summary of the results, ranging from optimality conditions to duality theory in continuous-time programming, as well as a fairly extensive list of relevant references, the reader is referred to \cite{monte:2019} and \cite{zalmai:1985}.

Linear continuous-time programming problems have also been treated in \cite{deOliveira2024,grinold:1969,reiland2:1980,Wu:2013,levinson:1966}. Regarding the nonlinear case, optimality conditions of the Karush-Kuhn-Tucker (KKT) type have been obtained by various authors, as can be seen in \cite{farr1:1974,reiland:1980,monte:2019,Vicanovic:2023a,Jovic:2022b}, to mention but a few.

In the context of nonlinear programming problems, the KKT optimality conditions are typically derived under the assumption of a certain regularity or constraint qualification condition. In the specific context of continuous-time optimization problems, Reiland \cite{reiland1:1980} employed a Zangwill-type condition, while Zalmai \cite{zalmai3:1985} utilized a Slater condition, for example. In recently published work, some classical constraint qualifications from mathematical programming have been adapted for the continuous-time context. In the papers by do Monte and de Oliveira \cite{monte:2019,monte:2020,monte:2021}, the authors established KKT optimality conditions under a linear independence constraint qualification, under a Mangasarian-Fromovitz constraint qualification, and under a constant rank constraint qualification, respectively.

The literature on numerical methods for solving this kind of problem includes contributions from Andreani et al. \cite{and:2004b}, Pullan \cite{pullan:1993}, Weiss \cite{weiss:2008}, Wen et al. \cite{wen:2012}, and Wu \cite{wu:2016}. In general, iterative methods that bring the solution closer through the KKT optimality conditions do not consider whether some constraint qualification is verified. In essence, a local minimizer may not be KKT, but it can always be approximated by a sequence of approximate-KKT points.

In this scenario, a distinct type of optimality condition emerges: the sequential optimality condition. A vector $x$ is said to satisfy the sequential optimality condition defined by a mathematical proposition $P$ if there exists a sequence $\{x^{k}\}$, which converges to $x$ and satisfies the condition $P(x^{k})$ for all values of $k$. Typically, a sequential optimality condition is linked to a specific quantity, denoted as $\varepsilon_k$, which is required to approach zero. The natural termination criterion associated with the sequential optimality condition suggests that the execution of the algorithm should be stopped when $\varepsilon_k$ reaches a sufficiently small value.

In the finite-dimensional case, Andreani et al. \cite{Andreani2010,and:2011} provide adequate theoretical tools to justify stopping criteria for nonlinear programming solvers. Two sequential optimality conditions were introduced, designated as approximate KKT (AKKT) condition and approximate gradient projection (AGP) condition. An analysis of the algorithmic consequences was conducted. Other authors have also explored this subject, with notable contributions from \cite{and:2017,and:2019,kan:2021,helou:2020}. In particular, the development of sequential optimality conditions and their extensions has been further advanced in \cite{Haeser2011, Andreani2022, kan:2021, helou:2020}, including extensions of AKKT to nonsmooth settings, variational inequalities, conic programming structures, cardinality-constrained problems, among others, thereby reinforcing their relevance in modern constrained optimization.

The concept of approximate (or asymptotic) KKT conditions was introduced into the infinite-dimensional framework by Kanzow et al. \cite{Kanzow:2018}, where the authors analyzed a nonlinear programming problem within Banach spaces. The results presented in reference \cite{Kanzow:2018} were generalized by Borgens et al. \cite{boorgens:2020}, who defined weak asymptotic KKT (w-AKKT) points and strong asymptotic KKT (s-AKKT) points. Three types of constraint qualifications, which are related to the asymptotic optimality conditions, are introduced. It is demonstrated that these are the weakest possible constraint qualifications that ensure that a given AKKT point of the optimization problem in Banach spaces is also a KKT point. Furthermore, the authors propose an augmented Lagrangian method that generates w-AKKT sequences under suitable assumptions.

The assumptions set forth in B\"orgens et al. \cite{boorgens:2020} require that the problem be formulated within a reflexive space. However, given the nature of the continuous-time programming problems that are posed within the space $L^\infty$, the results of \cite{boorgens:2020} become inapplicable.

The sequential type optimality conditions were also extended to optimal control problems with mixed constraints, as demonstrated in reference Moreira and de Oliveira \cite{Moreira2024}. For further details, please refer to \cite{Moreira2024}. While a continuous-time programming problem can be addressed as an optimal control problem, whereby the theory developed in \cite{Moreira2024} can be applied to continuous-time programming problems, the multipliers would be in the space $L^{1}$. This differs from the typical approach in the literature on continuous-time programming, where the multipliers are typically in the space $L^\infty$.

The present work is intended to serve two distinct purposes. The initial objective is to derive customized sequential type optimality conditions for continuous-time programming problems. We define AKKT sequences and AKKT solutions in the context of continuous-time programming and demonstrate that every local optimal solution is an AKKT one. It is noteworthy that no constraint qualification is required, which is a common feature of sequential type optimality conditions. In addition, we show that every AKKT solution is a global optimal solution if convexity assumptions are satisfied. The second objective is to propose a suitable augmented Lagrangian-type method for numerically solving continuous-time problems. The Lagrangian method is a well-established and reliable approach for solving nonlinear programming problems. The AKKT concept presented herein is employed as a stopping criterion for the augmented Lagrangian algorithm. This is a crucial point, as algorithms for nonlinear programming frequently neglect to consider the manner in which the algorithm terminates. In many instances, the methods result in a sequence $\{x^{k}\}$, $k\in\mathbb{N}$, where the constraint qualification is only verified at the limit. It is noteworthy that, to the best of our knowledge, there is currently no numerical method for continuous-time programming in the literature that is capable of dealing with nonlinear problems.

This paper is organized as follows. In Section \ref{prelim}, we define the continuous-time programming problem that this paper addresses, establish the necessary notation, and provide some fundamental definitions. Furthermore, we reformulate the continuous-time problem as an optimal control problem and demonstrate the relationship between the solution sets of the two problems. In Section \ref{AKKT}, we present the definitions of AKKT sequences and AKKT solutions, and we state and prove the main result. In Section \ref{ALM}, an augmented Lagrangian algorithm is described, and two theorems are provided that, under certain hypotheses, guarantee optimality and viability properties of sequential type for the limit points of the sequences generated by the algorithm, when they exist. Section \ref{appl} presents a few illustrative examples, which are employed for the purpose of conducting computational tests and subsequently analyzing the resulting data. Finally, in Section \ref{sec:Conclusion}, some concluding remarks are presented.

%%%%%%%%%%%%%%%%%%%%%%%%%%%%%%%%%%%%%%%%%%%%%%%%%%%%%%
\section{Preliminaries} \label{prelim}
%%%%%%%%%%%%%%%%%%%%%%%%%%%%%%%%%%%%%%%%%%%%%%%%%%%%%%

This study is focused on the following continuous-time programming problem that encompasses both equality and inequality constraints:
$$
\begin{array}{ll} 
\mbox{minimize} & P(x) = \displaystyle\int_{0}^{T} \phi(x(t),t)~ dt \\ 
\mbox{subject to} & h(x(t),t)= 0 ~ \mbox{a.e.} ~ t \in [0,T], \\
& g(x(t),t) \leq 0 ~ \mbox{a.e.} ~ t \in [0,T], \\
& x \in  L^{\infty}([0,T];\mathbb{R}^{n}),
\end{array}
\eqno{\text{(CTP)}}
$$
where $\phi : \mathbb{R}^{n} \times [0,T] \rightarrow \mathbb{R}$, $h = (h_1,h_2,\ldots,h_p) : \mathbb{R}^{n} \times [0,T] \rightarrow \mathbb{R}^{p}$ and $g = (g_1,g_2,\ldots,g_m) : \mathbb{R}^{n} \times [0,T] \rightarrow \mathbb{R}^m$ are given functions. The set of the equality and inequality constraints indices are denoted, respectively, by 
$$
I = \{1,2,\ldots,p\} \quad \text{and} \quad J = \{1,2,\ldots,m\}.% and N=\{1,2,\ldots, n\}.
$$

Throughout the paper, all integrals should be understood in the Lebesgue sense, and inequality signs between vectors should be read component-wise.

The feasible set of (CTP) is denoted by $\Omega$, that is,
$$
\Omega = \{ x \in L^{\infty}([0,T];\mathbb{R}^{n})  :  h(x(t),t) = 0, ~ g(x(t),t) \leq  0 ~ \mbox{a.e.} ~ t \in [0,T]\}.
$$

Let $r > 0$ be a real number and $\bar{x} \in \Omega$ be a feasible solution. For almost every $t \in [0,T]$, the following notations will be used throughout the paper:
\begin{equation*}
B_{r}(t) = \{ x \in \mathbb{R}^{n}  :  \| x-\bar{x}(t) \| < r \} \quad \text{and} \quad 
\bar{B}_{r}(t) = \{x \in \mathbb{R}^{n}  :  \| x-\bar{x}(t) \| \leq r \}.
\end{equation*}

\begin{definition} \label{local-global} 
A feasible solution $\bar{x} \in \Omega$ is said to be a local optimal solution for (CTP) if there exists $r > 0$ such that $P(\bar{x}) \leq P(x)$ for all $x \in \Omega$ such that $x(t)\in B_{r}(t)$ for almost every $t\in [0,T]$. In addition, $\bar{x} \in \Omega$ is said to be a global optimal solution if $P(\bar{x}) \leq P(x)$ for all $x \in \Omega$.
\end{definition}

\begin{definition}
The Lagrangian function $L : \mathbb{R}^{n} \times \mathbb{R}^{p} \times\mathbb{R}^{m} \times [0,T] \rightarrow \mathbb{R}$ associated to (CTP) is defined as 
$$
L(x,u,v,t) = \phi(x,t) + \sum_{i=1}^{p} u_{i}h_{i}(x,t) + \sum_{j=1}^{m} v_{j}g_{j}(x,t) ~ \mbox{a.e.} ~ t \in [0,T].
$$
\end{definition}

The basic assumptions will be said to be satisfied at $\bar{x} \in \Omega$ if there exists $r>0$ such that:
\begin{itemize}
\item[(H1)] $\phi(\cdot,t)$ is continuously differentiable a.e. $t \in [0,T]$;  
\item[] $\phi(x,\cdot)$ is Lebesgue measurable for each $x$;
%\item[] \textcolor{blue}{$\vert \phi(x,t) \vert + \Vert \nabla_x \phi(x,t) \Vert \leq c_{\phi}$ for all $x \in B_{r}(t)$ a.e. $t \in [0,T]$ for some $c_{\phi} > 0$.}
\item[] $\vert \phi(\bar{x}(\cdot),\cdot) \vert$ is integrable on $[0,T]$;
\item[] there exists an integrable function $c_\phi$ on $[0,T]$ such that $$\vert \phi(x,t) - \phi(y,t) \vert \leq c_{\phi}(t) \Vert x - y \Vert ~ \forall \, x,y \in \bar{B}_{r}(t) ~ \text{a.e.} ~ t \in [0,T];$$
\item[(H2)] $(h,g)(\cdot,t)$ is continuously differentiable a.e. $t \in [0,T]$; 
\item[] $(h,g)(x,\cdot)$ is Lebesgue measurable for each $x$; 
%\item[] \textcolor{blue}{$\Vert (h,g)(x,t) \Vert + \Vert \nabla_x (h,g)(x,t) \Vert \leq c_{h,g}$ for all $x \in B_{r}(t)$ a.e. $t \in [0,T]$ for some $c_{h,g} > 0$.}
\item[] $\Vert (h,g)(\bar{x}(\cdot),\cdot) \Vert$ is integrable on $[0,T]$;
\item[] there exists an integrable function $c_{h,g}$ on $[0,T]$ such that $$\Vert (h,g)(x,t) - (h,g)(y,t) \Vert \leq c_{h,g}(t) \Vert x - y \Vert ~ \forall \, x,y \in \bar{B}_{r}(t) ~ \text{a.e.} ~ t \in [0,T].$$
\end{itemize}
\medskip

\begin{remark}
It is evident that the assertions below are a direct consequence of (H1)--(H2):
\begin{itemize}
\item[(A1)] there exists an integrable function $k_{\phi}$ such that
$$
\vert \phi(x,t) \vert \leq k_{\phi}(t) ~ \forall \, x \in \bar{B}_{r}(t) ~ \mbox{a.e.} ~ t \in [0,T];
$$ 
\item[(A2)] there exists an integrable function $k_{h,g}$ such that
$$
\Vert (h,g)(x,t) \Vert \leq k_{h,g}(t) ~ \forall \, x \in \bar{B}_{r}(t) ~ \mbox{a.e.} ~ t \in [0,T].
$$
\end{itemize}
\end{remark}

This section will conclude with the definition of stationary solutions for the unconstrained problem.

\begin{definition} \label{stationary} 
Let $\bar{x} \in L^{\infty}([0,T];\mathbb{R}^{n})$ such that (H1) is valid. The solution $\bar{x} \in L^{\infty}([0,T];\mathbb{R}^{n})$ is said to be a stationary solution for the unconstrained problem
\begin{equation*} \label{UP}
\begin{array}{ll} 
\mbox{minimize} & P(x) = \displaystyle\int_{0}^{T} \phi(x(t),t)~ dt \\ 
\mbox{subject to} & x \in  L^{\infty}([0,T];\mathbb{R}^{n})
\end{array}
\end{equation*}
when 
$$
\nabla_x \phi(\bar{x}(t),t) = 0 ~ \text{a.e.} ~ t \in [0,T].
$$
\end{definition}

\section{Asymptotic Karush-Kuhn-Tucker Conditions} \label{AKKT}

This section is devoted to the definition of Karush-Kuhn-Tucker solutions of asymptotic type and the proof that every local optimal solution of (CTP) satisfies this definition.

\begin{definition} \label{Def_Seq_AKKT}
%Let $\bar{x} \in \Omega$ such that (H1) and (H2) are valid. 
A sequence $\{(x^{k},u^{k},v^{k})\}_{k \in \mathbb{N}} \subset L^{\infty}([0,T];\mathbb{R}^{n} \times \mathbb{R}^{p} \times\mathbb{R}^{m})$ is called an Asymptotic Karush-Kuhn-Tucker (AKKT) sequence for (CTP) %at $\bar{x} \in \Omega$ 
if
\begin{align}
& \int_0^T \nabla_{x}L(x^{k}(t),u^{k}(t),v^{k}(t),t) \cdot \gamma(t) ~ dt \to 0 ~ \forall \, \gamma \in L^{\infty}([0,T];\mathbb{R}^{n}), \label{2} \\ 
& v_{j}^{k}(t)g_{j}^{-}(x^k(t),t) \to 0 ~ \text{a.e.} ~ t \in [0,T], ~ j \in J, \label{3} \\
& v_j^k(t) \geq 0 ~ \text{a.e.} ~ t \in [0,T], ~ j \in J, \label{4}
\end{align}
where $g_j^{-}(x,t) := \max \{ -g_j(x,t),0 \}$ a.e. $t \in [0,T]$.
\end{definition}

\begin{definition}\label{pw-AKKT} 
Let $\bar{x} \in \Omega$ such that (H1) and (H2) are valid. The feasible solution $\bar{x} \in \Omega$ is said to be a pointwise asymptotic KKT (pw-AKKT) solution for (CTP) if there exists an AKKT sequence $\{(x^{k},u^{k},v^{k})\}_{k \in \mathbb{N}}$ for (CTP) %at $\bar{x}$ 
such that
$$
x^k (t) \rightarrow \bar{x}(t) ~ \mbox{ a.e.} ~ t \in [0,T].
$$
\end{definition}

\begin{theorem} \label{Teo_AKKT}
Let $\bar{x}$ be a local optimal solution for (CTP) and suppose that (H1)--(H2) are satisfied. Then $\bar{x}$ is a pw-AKKT solution for (CTP).
\end{theorem}

\begin{proof}
%We reduce $r > 0$, if necessary, such that $P(\bar{x}) \leq P(x)$ for all $x \in \Omega$ such that $x(t) \in B_{r}(t)$ a.e. $t \in [0,T]$, and the Hypotheses (H1)--(H3) are satisfied. 
%Let $\bar{\alpha}$ be given as in (\ref{8}). It follows from Proposition \ref{Prop1} that $(\bar{\alpha},\bar{x})$ is a weak local minimizer for (OCP). 

%\textcolor{red}{Let $\tilde{r} < \min\{r,r/k_{\phi}\}$, where $k_{\phi} > 0$ is the constant in (A1).}

The proof will proceed in a series of steps.

\noindent\textit{Application of the Ekeland Variational Principle.} We start by defining $f: \mathbb{R}^n \times [0,T] \rightarrow \mathbb{R}$ as
$$
f(x,t) = \max\{0,\vert h_1(x,t) \vert,\ldots,\vert h_p(x,t) \vert,g_1(x,t),\ldots,g_m(x,t)\} ~ \text{a.e.} ~ t \in [0,T].
$$
For the purpose of applying the Ekeland Variational Principle, we define $X = W^{1,1}([0,T];\mathbb{R}) \times W^{1,1}([0,T];\mathbb{R}) \times L^{\infty}([0,T];\mathbb{R}^{n}) \times \mathbb{R}$ and
\begin{align*}
    W = \left\{ (\alpha, \theta, x, e) \in X ~ 
    \left\vert ~ 
    \begin{array}{l} 
    \dot{\alpha}(t) = \phi(x(t),t) ~ \mbox{a.e.} ~ t \in [0,T],\\
    \dot{\theta}(t) = f(x(t),t) ~ \mbox{a.e.} ~ t \in [0,T], \\
    (\alpha(0),\theta(0)) = (0,0), \\
    x(t) \in \bar{B}_{\tilde{r}}(t) ~ \mbox{a.e.} ~ t \in [0,T], ~ e \in \mathbb{R} 
\end{array}\right.\right\},
\end{align*}
where $0<\tilde{r}<r$ and $r>0$ comes from the local optimality of $\bar{x}$. We also define $d_{W} : W \times W \rightarrow \mathbb{R}$ by
$$
d_{W} \big( (\alpha, \theta,x, e),(\alpha^{\prime}, \theta^{\prime},x^{\prime}, e^{\prime}) \big) = \vert e - e^{\prime} \vert + \int_{0}^{T} \Vert x(t)-x^{\prime}(t) \Vert ~dt.
$$
Finally, for each $k\in\mathbb{N}$, we define the function $\ell_{k} : W \rightarrow \bar{\mathbb{R}}$ as 
$$
\ell_{k}(\alpha, \theta, x, e) = \max\left\{e - \bar{\alpha}(T) + \dfrac{1}{k^{2}}, \vert \alpha(T) - e \vert \right\} + k\dist_{\mathbb{R}_{-}}(\theta(T)),
$$
where $\dist_{\mathbb{R}_{-}}(\theta)$ denotes the distance from $\theta$ to the set $\mathbb{R}_{-}$.

It is a simple matter to show that $(W,d_{W})$ is a complete metric space and $\ell_k$ is lower semicontinuous in $(W,d_{W})$ for all $k$.
    
Let us consider the following sequence of optimization problems:
\begin{equation} \label{Rk}
\begin{array}{ll} \mbox{minimize} & \ell_{k}(\alpha, \theta, x, e) \\ 
\mbox{subject to} & (\alpha,\theta,x,e) \in W.
\end{array}
\end{equation}
Let $\bar{\alpha} : [0,T] \rightarrow \mathbb{R}$ be given as
$$
\bar{\alpha}(t) = \int_0^t \phi(\bar{x}(t),t) \, dt ~ \text{a.e.} ~ t \in [0,T]
$$
and $\bar{\theta} : [0,T] \rightarrow \mathbb{R}$ be defined as $\bar{\theta}(t) \equiv 0$.
It is clear that $(\bar{\alpha}, \bar{\theta}, \bar{x}, \bar{\alpha}(T)) \in W$, so that $W \neq \emptyset$, and $\ell_{k}(\bar{\alpha}, \bar{\theta}, \bar{x}, \bar{\alpha}(T)) = 1/k^{2}$. Since $\ell_{k}$ is non-negative valued, it follows that $(\bar{\alpha}, \bar{\theta}, \bar{x}, \bar{\alpha}(T))$ is an ``$(1/k^{2})$-minimizer'' for the Problem \eqref{Rk}. By the Ekeland's Variational Principle \cite[Thm. 1.1]{Ekeland1974}, there exists a sequence $\{(\alpha^{k}, \theta^{k}, x^{k}, e^{k})\}_{k\in\mathbb{N}} \subset W$ such that, for each $k\in\mathbb{N}$,
\begin{align} 
\nonumber
    &\max\left\{e^{k} - \bar{\alpha}(T) + \dfrac{1}{k^{2}}, \vert\alpha^{k}(T) - e^{k}\vert\right\} + k\dist_{\mathbb{R}_{-}}(\theta^{k}(T)) \\
    & \hspace{0.5cm} \leq \max\left\{e - \bar{\alpha}(T) + \dfrac{1}{k^{2}}, \vert\alpha(T) - e \vert\right\} + k\dist_{\mathbb{R}_{-}}(\theta(T)) \nonumber \\
    & \hspace{1.0cm} + \frac{1}{k}d_{W}((\alpha, \theta, x, e),(\alpha^{k}, \theta^{k}, x^{k}, e^{k}))
\label{11}
\end{align} 
for all $(\alpha, \theta, x, e)\in W$ and also
\begin{align}\label{eq_dWP}
    d_{W}((\alpha^{k}, \theta^{k}, x^{k}, e^{k}), (\bar{\alpha}, \bar{\theta}, \bar{x}, \bar{\alpha}(T))) \leq \dfrac{1}{k} ~ \forall \, k \in \mathbb{N}.
\end{align}

The condition~\eqref{eq_dWP} implies that $e^{k} \to \bar{\alpha}(T)$ and that $x^{k}$ converges to $\bar{x}$ in the $L^1$ norm. By extracting a subsequence  (we keep the same index), we have that
\begin{align}
    x^k (t) \rightarrow \bar{x}(t) ~ \mbox{ a.e.} ~ t \in [0,T].
\label{conv_xk}
\end{align}
 
\noindent\textit{Application of the weak maximum principle.} 
Define the arc $\beta^{k} : [0,T] \rightarrow \mathbb{R}$ as $\beta^k(t) \equiv e^{k}$. Then $\beta^{k} \to \bar{\alpha}(T)$ uniformly.

We can express the minimization property \eqref{11} as follows: $(\alpha^{k}, \theta^{k}, x^{k}, \beta^{k})$ is a local minimizer for the following optimal control problem:
$$
\begin{array}{ll} 
    \mbox{minimize} & \max\left\{\beta(T) - \bar{\alpha}(T) + \dfrac{1}{k^{2}}, \vert\alpha(T) - \beta(T)\vert \right\} + k\dist_{\mathbb{R}_{-}}(\theta(T))\\
    & \quad + \dfrac{1}{k}\vert\beta(T) - \beta^{k}(T)\vert + \dfrac{1}{k}\displaystyle{\int_{0}^{T} \Vert x(t) - x^{k}(t) \Vert ~dt} \\ 
    \mbox{subject to} & \dot{\alpha}(t) = \phi(x(t),t) ~ \mbox{a.e.} ~ t \in [0,T], \\
    & \dot{\beta}(t) = 0 ~ \mbox{a.e.} ~ t \in [0,T], \\
    & \dot{\theta}(t) = f(x(t),t) ~ \mbox{a.e.} ~ t \in [0,T], \\
    & x(t) \in \bar{B}_{\tilde{r}}(t) ~ \mbox{a.e.} ~ t \in [0,T], \\
    & (\alpha(0),\beta(0),\theta(0)) \in \{0\}\times\mathbb{R}\times\{0\}, \\
    & (\alpha(T),\beta(T),\theta(T)) \in \mathbb{R}\times\mathbb{R}\times\mathbb{R}, \\
    & (\alpha, \beta, \theta, x) \in W^{1,1}([0,T];\mathbb{R}^3) \times  L^{\infty}([0,T];\mathbb{R}^{n}).
\end{array}
\eqno{(\text{R}_k)}
$$

The Hamiltonian function for ($\text{R}_k$) is given by 
$$
H_{k}(\alpha,\beta, \theta, \varphi, \chi, \psi,x, t) = \varphi \cdot \phi(x,t) + \chi \cdot 0 + \psi \cdot f(x,t) - \frac{1}{k}\Vert x - x^{k}(t) \Vert ~ \mbox{a.e.} ~ t \in [0,T].
$$
Applying the version of the weak maximum principle in \cite[Prop. 1]{dePinho2009}, for each $k \in \mathbb{N}$, we ensure the existence of ($\varphi^{k}, \chi^{k}, \psi^{k}) \in W^{1,1}([0,T];\mathbb{R}^{3})$ and $\zeta^{k} \in L^{1}([0,T];\mathbb{R}^{n})$, as well as an integrable function $\eta$, which does not depend on $k$, such that, 
\begin{align}
\nonumber
    &((-\dot{\varphi}^{k}(t), -\dot{\chi}^{k}(t), -\dot{\psi}^{k}(t)),\zeta^{k}(t)) \\
    &\hspace{1.0cm} \in  \mathrm{co} \left[\partial_{\alpha,\beta, \theta, x} H_{k}(\alpha^{k}(t), \beta^{k}(t), \theta^{k}(t), \varphi^{k}(t), \chi^{k}(t), \psi^{k}(t), x^{k}(t),t) \right], 
    \label{30} \\
    & \zeta^{k}(t) \in \eta(t) \partial d_{\bar{B}_{\tilde{r}}(t)}(x^{k}(t)) \subset N_{\bar{B}_{\tilde{r}}(t)}(x^{k}(t)),
    \label{31}
\end{align}   
for almost every $t \in [0,T]$, and
\begin{align}
\nonumber
    & (\varphi^{k}(0), \chi^{k}(0), \psi^{k}(0),-\varphi^{k}(T), -\chi^{k}(T), -\psi^{k}(T)) \nonumber \\
    & \hspace{0.6cm} \in \partial\left[\max\left\{\beta^{k}(T) - \bar{\alpha}(T) + \dfrac{1}{k^{2}}, \vert\alpha^{k}(T) - \beta^{k}(T)\vert\right\} + k\dist_{\mathbb{R}_{-}}(\theta^{k}(T)) \right] \nonumber \\
    & \hspace{1.0cm} + \left[\{0\}\times \{0\}\times\{0\}\times \{0\}\times\dfrac{1}{k}B \times \{0\}\right] \nonumber \\
    & \hspace{1.0cm} + N_{\{0\}\times\mathbb{R}\times\{0\}\times\mathbb{R}\times\mathbb{R}\times\mathbb{R}}(\alpha^{k}(0),\beta^{k}(0), \theta^{k}(0), \alpha^{k}(T), \beta^{k}(T), \theta^{k}(T)). \label{32}
\end{align}
for each $k \in \mathbb{N}$.
    
\noindent\textit{Derivation of the pw-AKKT sequence.} From (\ref{30}), by using properties of nonsmooth analysis (see Vinter  \cite{vinter:2010}), we obtain
\begin{align}
& (-\dot{\varphi}^{k}(t), -\dot{\chi}^{k}(t), -\dot{\psi}^{k}(t)), \zeta^{k}(t)) \nonumber \\
& \qquad \in \varphi^{k}(t)\nabla_{\alpha,\beta,\theta,x}\phi(x^k(t),t) + \psi^{k}(t)\, \mathrm{co}\, \partial_{\alpha,\beta,\theta,x} f(x^{k}(t),t) \nonumber \\
& \qquad \quad - \{0\} \times \{0\} \times \{0\} \times \frac{1}{k} B \label{Eq_1}
\end{align}
for each $k \in \mathbb{N}$ and almost every $t \in [0,T]$. It follows that $(-\dot{\varphi}^{k}(t), -\dot{\chi}^{k}(t), -\dot{\psi}^{k}(t)) = (0, 0, 0)$ for each $k\in\mathbb{N}$ and almost every $t\in [0,T]$. Therefore, $(\varphi^{k}, \chi^{k}, \psi^{k})$ is a constant function for each $k \in \mathbb{N}$. Furthermore, from (\ref{32}) we have that
\begin{align}
& (\varphi^{k}, \chi^{k}, \psi^{k},-\varphi^{k}, -\chi^{k}, -\psi^{k}) \nonumber \\
& \hspace{0.6cm} \in \partial\left[\max\left\{\beta^{k}(T) - \bar{\alpha}(T) + \dfrac{1}{k^{2}}, \vert\alpha^{k}(T) - \beta^{k}(T)\vert\right\}\right] + k \partial \dist_{\mathbb{R_-}}(\theta^{k}(T)) \nonumber \\
& \hspace{1.0cm} + \mathbb{R}\times \{0\}\times \mathbb{R}\times \{0\} \times\dfrac{1}{k}B \times \{0\}  ~ \forall \, k \in \mathbb{N}.
\label{Eq_revi_1}
\end{align}
Then, by employing the characterization of the subdifferential of the distance function by means of the normal cone (see \cite[Theorem 4.8.5]{vinter:2010}), we have
$$
-\psi^{k} \in k\left[ N_{\mathbb{R}_{-}}(\theta^{k}(T)) \cap B\right] = N_{\mathbb{R}_{-}}(\theta^{k}(T)) \cap kB ~ \forall \, k \in \mathbb{N}. 
$$
It follows that the limiting normal cone $N_{\mathbb{R}_{-}}(\theta^{k}(T))$ is nonempty and, consequently, $\theta^{k}(T) \leq 0$ for all $k \in \mathbb{N}$. Moreover, from the constraints in ($\text{R}_k$), since $\theta^{k} \in W^{1,1}([0,T]; \mathbb{R})$, we can write
$$
0 \geq \theta^{k}(T) = \theta^{k}(0) + \int_{0}^{T} f(x^{k}(t),t) dt = \int_{0}^{T}f(x^{k}(t),t) dt ~ \forall \, k \in \mathbb{N}.
$$
On the other hand, by its very definition, we know that $f(x^{k}(t),t) \geq 0$ for almost every $t \in [0,T]$, for all $k \in \mathbb{N}$. Therefore, $f(x^{k}(t),t) = 0$ for almost every $t \in [0,T]$, for all $k \in \mathbb{N}$. Again by its definition, we obtain
$$
\max\{0,\vert h_{1}(x^{k}(t), t) \vert,\ldots,\vert h_{p}(x^{k}(t), t) \vert,g_{1}(x^{k}(t), t), \ldots, g_{m}(x^{k}(t), t)\} = 0
$$
for almost every $t \in [0, T]$, for all $k \in \mathbb{N}$. This implies that
\begin{equation} \label{viability}
    h(x^{k}(t),t) = 0 \quad \text{and} \quad g(x^{k}(t),t) \leq 0 \text{ a.e. } t \in [0, T] ~ \forall \, k \in \mathbb{N}.
\end{equation}
Let us define, for each $k \in \mathbb{N}$, $\tilde{\ell}_k : \mathbb{R} \times \mathbb{R} \rightarrow \mathbb{R}$ as
$$
\tilde{\ell}_k(\alpha,\beta) = \max\left\{\beta - \bar{\alpha}(T) + \dfrac{1}{k^{2}}, \vert \alpha - \beta \vert \right\}.
$$
We will show that, for each $k \in \mathbb{N}$,
\begin{equation}
\tilde{\ell}_{k}(\alpha^{k}(T), \beta^{k}(T)) = \max\left\{\beta^{k}(T) - \bar{\alpha}(T) + \dfrac{1}{k^{2}}, \vert\alpha^{k}(T) - \beta^{k}(T)\vert\right\} > 0.
\label{Eq_tildelk}
\end{equation}
If not, $\tilde{\ell}_{k}(\alpha^{k}(T), \beta^{k}(T)) = 0$ for some $k \in \mathbb{N}$, since $\ell_{k}$ is nonnegative. But then
$$
\beta^k(T) - \bar{\alpha}(T) + \dfrac{1}{k^2} \leq 0 \quad \text{and} \quad \alpha^{k}(T) = \beta^{k}(T),
$$
so that
$$
\alpha^{k}(T) \leq \bar{\alpha}(T) - \dfrac{1}{k^{2}}.
$$
From the fact that $\bar{x},x^k \in W$, the inequality above can be rewritten as
$$
\int_0^T \phi(x^k(t),t) \, dt \leq \int_0^T \phi(\bar{x}(t),t) \, dt - \dfrac{1}{k^2},
$$
that is,
$$
P(x^k) \leq P(\bar{x}) - \dfrac{1}{k^2}.
$$
Note also that, since $x^k \in W$, $x^k(t) \in \bar{B}_{\tilde{r}}(t) \subset B_{r}(t)$ for almost every $t \in [0,T]$. These, together with (\ref{viability}), violate the optimality of $\bar{x}$ for (CTP).

The following estimate for $\partial \tilde{\ell}_{k}(\beta^{k}(T), \theta^{k}(T))$ is verified:
\begin{equation}
\partial\tilde{\ell}_{k}(\beta^{k}(T), \theta^{k}(T)) \subset \left\{ (a^{k}, b^{k} - a^{k}) \in \mathbb{R}\times\mathbb{R} : b^{k}\geq 0 , ~ b^{k} + \vert a^{k} \vert = 1 \right\}.
\label{Eq_rco_1}
\end{equation}
There are two cases to consider.
\begin{itemize}
    \item[(a)] $\alpha^{k}(T) = \beta^{k}(T)$. In this case, by continuity, it follows from \eqref{Eq_tildelk} that 
    $$
        \tilde{\ell}_{k}(\alpha, \beta) = \beta - \bar{\alpha}(T) + \dfrac{1}{k^{2}}
    $$
    for all $(\alpha, \beta)$ in a certain neighborhood of $(\alpha^{k}(T), \beta^{k}(T))$. Consequently, \eqref{Eq_rco_1} is true with $a^{k} = 0$ and $b^k = 1$.
   \item[(b)] $\alpha^{k}(T) \neq \beta^{k}(T)$.  In this case, the estimate \eqref{Eq_rco_1} will follow from the max rule (see \cite[Theorem 5.5.2]{vinter:2010}). First, note that, from the chain rule (see \cite[Theorem 5.5.1]{vinter:2010}),
    $$
        \partial |\alpha - \beta| \big|_{(\alpha,\beta)=(\alpha^{k}(T), \beta^{k}(T))} \subset \{(\tilde{a}, -\tilde{a}) : |\tilde{a}| = 1\}.
    $$
    By the max rule, we have
    \begin{align*}
        \partial\tilde{\ell}_{k}(\beta^{k}(T), \theta^{k}(T)) &\subset \{(1-\lambda^{k})(0,1)+\lambda^{k}(\tilde{a},-\tilde{a}) ~:~ \lambda^{k} \in [0, 1],~ |\tilde{a}|=1\} \\
        &= \{(\lambda^{k} \tilde{a}, 1-\lambda^{k}-\lambda^{k} \tilde{a}) ~:~ \lambda^{k} \in [0, 1],~ |\tilde{a}|=1\} ~ \forall \, k \in \mathbb{N}.
    \end{align*}    
    Setting $a^{k} =\lambda^{k} \tilde{a}$ and $b^{k}=1-\lambda^{k}$ for all $k \in \mathbb{N}$, yields $b^{k}\ge0$ and $b^{k}+|a^{k}|=1$ for all $k \in \mathbb{N}$, as desired.	
\end{itemize}

From \eqref{Eq_revi_1} and \eqref{Eq_rco_1}, we know that
%for all $k \in \mathbb{N}$,
%%
%\begin{align}
%    (\varphi^{k}, \psi^{k}) \in \mathbb{R}\times\mathbb{R}
%\label{Eq_psi_R}
%\end{align}
%%
%and
%
\begin{equation} \label{eq_nova1}
(\chi^{k},-\varphi^{k}, -\chi^{k}, -\psi^{k}) = \left(0, a^{k}, b^{k} - a^{k} + \dfrac{1}{k}\xi, \vartheta^{k}\right)
\end{equation}
with $b^k \geq 0$, $b^{k}+\vert a^{k} \vert = 1$, $\xi \in B$ and $\vartheta^{k} \in N_{\mathbb{R}_{-}}(\theta^{k}(T)) \cap kB = \mathbb{R}_{+} \cap kB$ for all $k \in \mathbb{N}$. Then,
\begin{equation} \label{eq_nova2}
0 = - \chi^{k} = b^{k} - a^{k} + \dfrac{1}{k}\xi
\end{equation}
and
\begin{equation}
    1 = b^{k}+|a^{k}| = b^{k} + \left\vert b^{k} + \dfrac{1}{k}\xi \right\vert ~ \forall \, k \in \mathbb{N}.
\label{Eq_bk}
\end{equation}
Since $\{b^{k}\}$ is bounded, by taking a subsequence (which we do not relabel), we may conclude that $b^{k} \to b$. Thus, passing to the limit in \eqref{Eq_bk}, it follows that
$$
b + \vert b \vert = 1 \quad \Rightarrow \quad b = \dfrac{1}{2}. 
$$
From \eqref{eq_nova1} and \eqref{eq_nova2}, we see that
$$ 
- \varphi^{k} = a^{k} = b^{k} + \dfrac{1}{k}\xi,
$$
which yields, after taking limits, to
$$
\varphi^{k} \to \varphi := -\dfrac{1}{2}. 
$$

By \eqref{eq_nova1} we know that $-\psi^{k} = \vartheta^{k}$ for all $k \in \mathbb{N}$. Using this, it follows from the last argument in (\ref{Eq_1}), after applying the max rule and an appropriate measurable selection theorem, that there exist  a sequence $\{\sigma^k\}_{k\in\mathbb{N}} \subset \mathbb{R}$ and sequences of measurable functions $\{\tilde{u}^k\}_{k\in\mathbb{N}}$, $\{\tilde{v}^k\}_{k\in\mathbb{N}}$ and $\{\mu^{k}\}_{k\in\mathbb{N}}$ satisfying
\begin{align}
\nonumber
    \zeta^{k}(t) = & ~ \varphi^{k}\nabla_{x}\phi(x^k(t),t) - \sum_{i=1}^{p} [\vartheta^{k}\tilde{u}_{i}^k(t)\sigma^k] \nabla_{x} h_{i}(x^{k}(t),t) \\
    &~ - \sum_{j=1}^{m} [\vartheta^{k}\tilde{v}_j^k(t)] \nabla_{x} g_{j}(x^{k}(t),t) - \mu^{k}(t) ~ \text{a.e.} ~ t \in [0,T] ~ \forall \, k \in \mathbb{N},
\label{Eqrevvv_1}
\end{align}
where $\sigma^k \in [-1,1]$, $\mu^{k}(t) \in (1/k)B$ and
\begin{align}
& \tilde{u}_i^k(t) \geq 0, ~ \tilde{v}_j^k(t) \geq 0 ~ \text{a.e.} ~ t \in [0,T] ~ \forall \, k \in \mathbb{N}, ~ i \in I, ~ j \in J, \label{nonneg} \\
& \sum_{i=1}^p \tilde{u}_i^k(t) + \sum_{j=1}^k \tilde{v}_j^k(t) \leq 1 ~ \text{a.e.} ~ t \in [0,T] ~ \forall \, k \in \mathbb{N}, \label{soma1} \\
& \tilde{v}_j^k(t) g_j(x^k(t),t) = 0 ~ \text{a.e.} ~ t \in [0,T] ~ \forall \, k \in \mathbb{N}, ~~ j \in J. \label{compl}
\end{align}
It is clear that $\mu^{k} \to 0$ in $L^{\infty}([0, T]; \mathbb{R}^{n})$. Since $\varphi^{k} \to \varphi = -\tfrac{1}{2}$, for sufficiently large $k$, $\varphi^k < 0$ and division by $-\varphi^{k}$ is well defined. We then discard the initial terms of all considered sequences. Dividing both sides of \eqref{Eqrevvv_1} by $-\varphi^{k}$ yields
\begin{align}
\nonumber
    \tilde{\zeta}^{k}(t) = & ~ - \nabla_{x}\phi(x^k(t),t)
    - \sum_{i=1}^{p} \frac{\vartheta^{k}\sigma^k}{(-\varphi^{k})}\tilde{u}_{i}^k(t)\nabla_{x} h_{i}(x^{k}(t),t) \\
    & ~ - \sum_{j=1}^{m} \frac{\vartheta^{k}}{(-\varphi^{k})}\tilde{v}_{j}^k(t)\nabla_{x} g_{j}(x^{k}(t),t)
    + \tilde{\mu}^{k}(t) ~ \text{a.e. } t \in [0,T] ~ \forall \, k \in \mathbb{N},
\label{eqtn}
\end{align}
where, for each $k \in \mathbb{N}$, $\tilde{\zeta}^{k}(t) := \zeta^{k}(t)/(-\varphi^{k})$, $\tilde{\mu}^{k}(t) := \mu^{k}(t)/\varphi^{k}$ for almost every $t \in [0,T]$. Clearly, $\tilde{\mu}^{k} \to 0$ in $L^{\infty}([0,T]; \mathbb{R}^{n})$. Moreover, from \eqref{31},
\begin{equation} \label{Eq_rev_31i}
\tilde{\zeta}^{k}(t) \in \dfrac{\eta(t)}{(-\varphi^{k})} \partial d_{\bar{B}_{\tilde{r}}(t)}(x^{k}(t))
\end{equation} 
and
\begin{equation} \label{Eq_rev_31ii}
\tilde{\zeta}^{k}(t) \in N_{\bar{B}_{\tilde{r}}(t)}(x^{k}(t)) ~ \text{a.e. } t \in [0,T] ~ \forall \, k \in \mathbb{N}.
\end{equation}

Let us define, for all $k \in \mathbb{N}$ and almost every $t \in [0,T]$,  
\begin{align}
    & \varepsilon^{k}(t) := \tilde{\mu}^{k}(t) - \tilde{\zeta}^{k}(t), \label{epsilonk} \\
    & u_{i}^{k}(t) := \dfrac{\vartheta^{k}\sigma^k}{(-\varphi^{k})}\tilde{u}_{i}^k(t), ~ i \in I, \label{uik} \\
    & v_{j}^{k}(t) := \dfrac{\vartheta^{k}}{(-\varphi^{k})}\tilde{v}_{j}^k(t), ~ j \in J. \label{vjk}
\end{align}

It is clear from (\ref{nonneg}) and (\ref{soma1}) that $\{(u^k,v^k)\}_{k \in \mathbb{N}} \subset L^\infty([0,T];\mathbb{R}^p \times \mathbb{R}^m)$. Also, it follows from \eqref{eqtn} that
\begin{equation} 
    \nabla_{x}L(x^{k}(t),u^{k}(t),v^{k}(t),t) = \varepsilon^{k}(t) ~ \text{a.e.} ~ t \in [0,T] ~ \forall \, k \in \mathbb{N}.
\label{eqLag}
\end{equation}

It is well known that the distance function is Lipschitz with constant equal to 1. Then, as indicated in \cite[Prop. 4.3.3]{vinter:2010}, $\partial d_{\bar{B}_{r}(t)}(x^{k}(t)) \subset \bar{B}_{1}$. Hence, from (\ref{Eq_rev_31i}) we see that $\vert\tilde{\zeta}^{k}(t)\vert \leq \vert\eta(t)\vert/\vert\varphi^{k}\vert$ for almost every $t \in [0, T]$ and for all $k \in \mathbb{N}$. Note that $\{\varphi^{k}\}$ bounded, because it is convergent. Therefore, $\{\tilde{\zeta}^{k}\}_{k\in \mathbb{N}}$ is uniformly integrably bounded, and by the Dunford-Pettis Theorem \cite[Thm. 2.5.1]{vinter:2010}, we can extract a subsequence (we do not relabel) such that $\tilde{\zeta}^{k} \rightharpoonup \zeta$, with $\zeta \in L^{1}([0, T];\mathbb{R}^{n})$. Note also that $x^k \rightharpoonup \bar{x}$ in $L^{1}([0, T];\mathbb{R}^{n})$, since $x^k \to \bar{x}$ in $L^{1}([0, T];\mathbb{R}^{n})$ (by \cite[Prop. 3.5]{brezis:2011}). Thus, appealing to the properties of upper semicontinuity of limiting normal cones, from (\ref{Eq_rev_31ii}), we obtain
\begin{equation*}
\zeta(t) \in N_{\bar{B}_{\tilde{r}}(t)}(\bar{x}(t)) = \{ 0 \} ~ \text{a.e.} ~ t \in [0,T].
\end{equation*}
Consequently, $\tilde{\zeta}^{k} \rightharpoonup 0$ in $L^{1}([0,T];\mathbb{R}^{n})$. Furthermore, since $\tilde{\mu}^{k} \to 0$ in \linebreak $L^{\infty}([0, T]; \mathbb{R}^{n})$, $\tilde{\mu}^{k} \to 0$ in $L^{1}([0,T];\mathbb{R}^{n})$ and $\tilde{\mu}^{k} \rightharpoonup 0$ in $L^{1}([0,T];\mathbb{R}^{n})$. 
Therefore, we see from \eqref{epsilonk} that $\varepsilon^{k} \rightharpoonup 0$ in $L^{1}([0,T];\mathbb{R}^{n})$. Hence, from (\ref{eqLag}),
$$
\int_0^T \nabla_{x}L(x^{k}(t),u^{k}(t),v^{k}(t),t) \cdot \gamma(t) ~ dt \to 0 ~ \forall \, \gamma \in L^{\infty}([0,T];\mathbb{R}^{n}),
$$
that is, (\ref{2}) is satisfied.

From (\ref{viability}), (\ref{compl}) and (\ref{vjk}), we have
\begin{align*}
    v_j^k(t)g_j^-(x^k(t),t) & = \dfrac{\vartheta^{k}}{(-\varphi^{k})} \tilde{v}_j^k(t) \max \{ -g_j(x^k(t),t),0 \} \\
    & = \dfrac{\vartheta^{k}}{(-\varphi^{k})} \tilde{v}_j^k(t) [-g_j(x^k(t),t)] \\
    & = 0 ~ \text{a.e.} ~ t \in [0,T] ~ \forall \, k \in \mathbb{N}, ~ j \in J,
\end{align*}
so that (\ref{3}) is verified. 

As seen above, $\varphi^k < 0$ and $\vartheta^{k} \in N_{\mathbb{R}_{-}}(\theta^{k}(T)) \cap kB = \mathbb{R}_{+} \cap kB$, so that $\vartheta^k \geq 0$ for all $k \in \mathbb{N}$. Then, it follows directly from \eqref{nonneg} and \eqref{vjk} that (\ref{4}) is satisfied as well. 

Therefore, we see that the sequence $\{(x^k,u^k,v^k)\}_{k \in \mathbb{N}}$ is an AKKT sequence according to Definition \ref{Def_Seq_AKKT}. From (\ref{conv_xk}) we can conclude that $\bar{x}$ is a pw-AKKT solution for (CTP). 

\end{proof}
\medskip

The preceding theorem indicates that pw-AKKT can be regarded as a genuine necessary optimality condition. It should be noted that no constraint qualification is imposed. Although the primary purpose of the pw-AKKT conditions is to validate stop criteria in numerical methods of optimization in the continuous-time context, they may also be useful for characterizing optimal solutions in situations where the classical KKT conditions are ineffective. This is demonstrated in the subsequent example. 

\begin{example}
We consider the continuous-time problem below:
$$
\begin{array}{ll}
\text{minimize} & P(x) = \displaystyle{\int_0^1 \left( t - \frac{1}{2} \right) x_1(t) ~ dt} \\
\text{subject to} & - \left( t - \frac{1}{2} \right) [x_1(t)]^3 + x_2(t) \leq 0 ~ \text{a.e.} ~ t \in [0,1], \\
& - x_2(t) \leq 0 ~ \text{a.e.} ~ t \in [0,1], \\
& x \in L^\infty([0,1];\mathbb{R}^2).
\end{array}
$$
It is easy to see that $\bar{x}(t) := (0,0)$ a.e. in $[0,1]$ is an optimal solution. Moreover, since
$$
\nabla_x \phi(\bar{x}(t),t) = \begin{bmatrix} t-\frac{1}{2} \\ 0 \end{bmatrix}, ~
\nabla_x g_1(\bar{x}(t),t) = \begin{bmatrix} 0 \\ 1 \end{bmatrix}, ~
\nabla_x g_2(\bar{x}(t),t) = \begin{bmatrix} 0 \\ -1 \end{bmatrix} ~ \text{a.e.} ~ t \in [0,1],
$$
the KKT conditions never hold. On the other hand, by defining
$$
x_1^k(t) = \left( t - \frac{1}{2} \right) \frac{1}{k}, ~ x_2^k(t) = 0, ~
v_1^k(t) = v_2^k(t) = \frac{k^2}{3 \left( t - \frac{1}{2} \right)^2} ~ \text{a.e.} ~ t \in [0,1] ~ \forall k \in \mathbb{N},
$$
we see that $\bar{x}$ is a pw-AKKT solution. Indeed, 
\begin{align*}
& \nabla_x \phi(x^k(t),t) + v_1^k(t) \nabla_x g_1(x^k(t),t) + v_2^k(t) \nabla_x g_2(x^k(t),t) \\
& \qquad = \begin{bmatrix} t-\frac{1}{2} \\ 0 \end{bmatrix}
        + v_1^k(t) \begin{bmatrix} -3 \left( t - \frac{1}{2} \right) [x_1^k(t)]^2 \\ 1 \end{bmatrix}
        + v_2^k(t) \begin{bmatrix} 0 \\ -1 \end{bmatrix} \\
& \qquad = 
\begin{bmatrix}
\left( t - \frac{1}{2} \right) -3 v_1^k(t) \left( t - \frac{1}{2} \right) [x_1^k(t)]^2 \\ v_1^k(t) - v_2^k(t)    
\end{bmatrix} \\
& \qquad = \begin{bmatrix} 0 \\ 0 \end{bmatrix} ~ \text{a.e.} ~ t \in [0,1] ~ \forall k \in \mathbb{N}.
\end{align*}
\end{example}

Subsequently, it is shown that, under convexity assumptions, AKKT also constitutes a sufficient optimality condition.

\begin{theorem} \label{Teo_Suf}
Let $\bar{x} \in \Omega$ be an pw-AKKT solution and $\{(x^{k}, u^{k}, v^{k})\}_{k \in \mathbb{N}}$ be an associated  AKKT sequence. Assume that (H1) and (H2) are satisfied, $\{x^k\}_{k \in \mathbb{N}}$ is a bounded sequence in $L^\infty([0,T];\mathbb{R}^n)$, and
\begin{equation} \label{Hip_eq_1}
\sum_{i=1}^{p} u_{i}^{k}(t) h_{i}(x^{k}(t),t) + \sum_{j=1}^{m} v_{j}^{k}(t) g_{j}(x^{k}(t),t) \geq 0 ~ \text{a.e.} ~ t \in [0,T] ~ \forall k \in \mathbb{N}.
\end{equation}
If, for almost every $t \in [0,T]$, $\phi(\cdot,t)$, $g_{j}(\cdot,t)$, $j = 1,\ldots,m$, are convex functions and $h_{i}(\cdot,t)$, $i = 1,\ldots,p$, are affine functions, then $\bar{x}$ is a global optimal solution. 
\end{theorem}

\begin{proof}
Let $x \in \Omega$. Since, by hypothesis, $\phi$, $g_{j}$, $j = 1,\ldots,m$, are convex functions and $h_{i}$, $i = 1,\ldots,p$, are affine functions, it follows, for almost every $t \in [0,T]$ and for all $k \in \mathbb{N}$, that
\begin{align}
& \phi(x(t),t) \geq \phi(x^{k}(t),t) + \nabla_{x} \phi(x^{k}(t),t) \cdot (x(t) - x^{k}(t)), \label{Eqqq_17} \\
& g_{j}(x(t),t) \geq g_{j}(x^{k}(t),t) + \nabla_{x} g_{j}(x^{k}(t),t) \cdot (x(t) - x^{k}(t)), ~ j = 1,\ldots,m, \label{Eqqq_18} \\
& h_{i}(x(t),t) = h_{i}(x^{k}(t),t) + \nabla_{x} h_{i}(x^{k}(t),t) \cdot (x(t) - x^{k}(t)), ~ i = 1,\ldots,p. \label{Eqqq_19}
\end{align}
Thus, from \eqref{Eqqq_17}, \eqref{Eqqq_18}, and \eqref{Eqqq_19}, we obtain
\begin{align*}
\phi(x(t),t) & + \sum_{i=1}^{p} u_{i}^{k}(t) h_{i}(x(t),t) + \sum_{i=1}^{m} v_{j}^{k}(t) g_{j}(x(t),t) \\
& \geq \phi(x^{k}(t),t) + \sum_{i=1}^{p} u_{i}^{k}(t) h_{i}(x^{k}(t),t) + \sum_{i=1}^{m} v_{j}^{k}(t) g_{j}(x^{k}(t),t) \\
& \quad + \nabla_{x} L(x^{k}(t),u^{k}(t),v^{k}(t),t) \cdot (x(t) - x^{k}(t)) ~ \text{a.e.} ~ t \in [0,T] ~ \forall k \in \mathbb{N}.
\end{align*}
From (\ref{4}), we know that $v_{j}^{k}(t) \geq 0$ for almost every $t \in [0,T]$, $j = 1,\ldots,m$. Then, from $x \in \Omega$, we get
$$
\sum_{i=1}^{p} u_{i}^{k}(t) h_{i}(x(t),t) = 0 \quad \text{and} \quad \sum_{i=1}^{m} v_{j}^{k}(t) g_{j}(x(t),t) \leq 0 ~ \text{a.e.} ~ t \in [0,T] ~ \forall k \in \mathbb{N}.
$$
Therefore, by (\ref{Hip_eq_1}), we obtain, for almost every $t \in [0,T]$ and for all $k \in \mathbb{N}$,
\begin{align*}
\phi(x(t),t) \geq \phi(x^{k}(t),t) + \nabla_{x} L(x^{k}(t),u^{k}(t),v^{k}(t),t) \cdot (x(t) - x^{k}(t)).
\end{align*}
Integrating both sides from $0$ to $T$,
\begin{align*}
\int_{0}^{T} \phi(x(t),t) ~dt \geq & \int_{0}^{T}\phi(x^{k}(t),t) ~dt \\ & + \int_{0}^{T} \nabla_{x} L(x^{k}(t),u^{k}(t),v^{k}(t),t) \cdot (x(t) - x^{k}(t)) ~dt
\end{align*}
for all $k \in \mathbb{N}$. From Definition \ref{Def_Seq_AKKT}, we know that $\nabla_{x} L(x^{k}(\cdot),u^{k}(\cdot),v^{k}(\cdot),\cdot) \rightharpoonup 0$ in $L^1([0,T];\mathbb{R}^n)$. By hypothesis, we have that $\{x(t) - x^{k}(t)\}_{k \in \mathbb{N}}$ is bounded in $L^\infty([0,T];\mathbb{R}^n)$, and from Definition \ref{pw-AKKT}, we know that $x(t) - x^{k}(t) \to x(t) - \bar{x}(t)$ for almost every $t \in [0,T]$. Hence, by \cite[Prop. 18]{Laurencot:2015},
$$
\lim_{k \to \infty} \int_{0}^{T} \nabla_{x} L(x^{k}(t),u^{k}(t),v^{k}(t),t) \cdot (x(t) - x^{k}(t)) ~dt = 0,
$$
so that
$$
\int_{0}^{T}\phi(x(t),t) ~dt \geq \lim_{k \to \infty} \int_{0}^{T} \phi(x^{k}(t),t) ~dt.
$$
Therefore, by the Dominated Convergence Theorem (see, for example, \cite[Thm. 3.25]{Gordon1994}), we conclude the result. 
\end{proof}

In regard to the assumption that the sequence $\{x^k\}_{k \in \mathbb{N}}$ is bounded as stated in the preceding theorem, we note that the sequence $\{x^k\}_{k \in \mathbb{N}}$ constructed in the proof of Theorem \ref{Teo_AKKT} has this property.

%Additionally, we note that in the proof of Theorem \ref{Teo_Suf}, conditions (\ref{Eqqq_17})--(\ref{Eqqq_19}) (related to the convexity assumptions) may be modified as follows:
%\begin{align*}
%& \phi(x(t),t) \geq \phi(x^{k}(t),t) + \nabla_{x} \phi(x^{k}(t),t) \cdot \eta(t) ~ \text{a.e.} ~ t \in [0,T], \\
%& g_{j}(x(t),t) \geq g_{j}(x^{k}(t),t) + \nabla_{x} g_{j}(x^{k}(t),t) \cdot \eta(t) ~ \text{a.e.} ~ t \in [0,T], ~ j = 1,\ldots,m, \\
%& h_{i}(x(t),t) = h_{i}(x^{k}(t),t) + \nabla_{x} h_{i}(x^{k}(t),t) \cdot \eta(t) ~ \text{a.e.} ~ t \in [0,T], ~ i = 1,\ldots,p,
%\end{align*}
%for some $\eta(t) = \eta(x(t),x^k(t))$ a.e. $t \in [0,T]$, $k \in \mathbb{N}$. The conditions above can be classified as a form of generalized convexity, which is referred to in the literature as invexity.

\section{The Augmented Lagrangian Method} \label{ALM}

This section presents the augmented Lagrangian method, accompanied by a concise examination of its viability and optimality.

For each $\rho>0$, let us define the augmented Lagrangian function $L_{\rho} : \mathbb{R}^{n} \times \mathbb{R}^{p} \times \mathbb{R}_{+}^{m} \times [0,T] \rightarrow \mathbb{R}$ as
\begin{align*}
L_{\rho}(x,u,v,t) := \phi(x,t) & + \frac{\rho}{2} \sum_{i=1}^{p} \left[ h_{i}(x,t) + \frac{u_{i}}{\rho}\right]^2 \\ 
& + \frac{\rho}{2} \sum_{j=1}^{m} \left[ \max \left\{ 0,g_{j}(x,t) + \frac{v_{j}}{\rho} \right\} \right]^2.
\end{align*}

\begin{algorithm} \caption{Augmented Lagrangian Method} 
%\vskip6pt

Let $M,~N$, $\gamma >1$, $\rho_{1}>0$ and $\tau\in (0,1)$. Let $\{\delta_k\}_{k \in \mathbb{N}} \subset \mathbb{R}_+$ such that $\delta_k \to 0$.

Choose $(\tilde{u}^{1},\tilde{v}^{1}) \in L^{\infty}([0,T]; \mathbb{R}^{p}) \times L^{\infty}([0, T]; \mathbb{R}_{+}^{m})$ satisfying
$$
\tilde{u}^{1}_{i}(t) \in [-M,M] ~ \text{a.e.} ~ t \in [0,T], ~ i \in I, ~
\tilde{v}^{1}_{j}(t) \in [0,N] ~ \text{a.e.} ~ t \in [0,T], ~ j \in J.
$$

Choose an arbitrary initial solution $x^{0}\in L^{\infty}([0,T];\mathbb{R}^{n})$. 

Define $V_{j}^{0}(t) := \max \{ g_{j}(x^{0}(t),t),0 \} ~ \text{a.e.} ~ t \in [0,T], ~ j \in J$.

Set $x^1 = x^0$ and $(u^{1},v^{1}) = (\tilde{u}^{1},\tilde{v}^{1})$.
    
Set $k=1$.

\begin{algorithmic}
%\State \textbf{Step 1.} If $(x^{k}, u^{k}, v^{k})$ satisfies a suitable termination criterion: STOP.
\State \textbf{Step 1.} If conditions below are fulfilled: STOP.
\begin{align*}
& \left\vert \int_0^T \nabla_x L(x^{k}(t),u^{k}(t),v^{k}(t)) \cdot \gamma(t) ~ dt \right\vert \leq \delta_k ~ \forall \, \gamma \in L^{\infty}([0,T];\mathbb{R}^n), \\
& \vert v_{j}^{k}(t)g_{j}^{-}(x^k(t),t) \vert \leq \delta_k ~ \text{a.e.} ~ t \in [0,T], ~ j \in J, \\
& v_j^k(t) \geq 0 ~ \text{a.e.} ~ t \in [0,T], ~ j \in J.
\end{align*} 

\State \textbf{Step 2.} Compute an approximate solution $x^{k}\in L^{\infty}([0,T];\mathbb{R}^{n})$ of 
%
%\begin{equation*} \label{Eq_sup_p}
$$
\text{minimize} ~ \int_{0}^{T} L_{\rho_{k}}(x(t),\tilde{u}^{k}(t),\tilde{v}^{k}(t),t) ~ dt ~ \text{over} ~ x \in L^\infty([0,T];\mathbb{R}^n).
\eqno{(\text{P}_k)}
$$
%\end{equation*}
%    
\State \textbf{Step 3.} Define, for each $k \in \mathbb{N}$, $V^{k} \in L^{\infty}([0,T],\mathbb{R}^{m})$ and $H^k \in L^{\infty}([0,T],\mathbb{R}^{n})$, respectively, as
\begin{align*}
& V_{j}^{k}(t) := \max\{g_{j}(x^{k}(t),t),-\tilde{v}_{j}^{k}(t)/\rho_{k}\} ~ \mbox{a.e.} ~ t \in [0,T], ~ j \in J, \\
& H^k(t) := h(x^k(t),t) ~ \mbox{a.e.} ~ t \in [0,T].
\end{align*}
\State \quad \textbf{If} $\max\{\|H^k\|_{\infty},\|V^{k}\|_{\infty}\} \leq \tau \max\{\|H^{k-1}\|_{\infty},\|V^{k-1}\|_{\infty}\}$, define $\rho_{k+1}=\rho_{k}$. 
\State \quad \textbf{Else}, $\rho_{k+1}=\gamma\rho_{k}$.
\State \textbf{Step 4.} Calculate 
\begin{align*}
& u_{i}^{k}(t) = \tilde{u}_{i}^{k}(t) + \rho_{k}h_{i}(x^{k}(t),t) ~ \mbox{a.e.} ~ t \in [0,T], ~ i \in I, \\
& v_{j}^{k}(t) = \max\{\tilde{v}_{j}^{k}(t)+\rho_{k}g_{j}(x^{k}(t),t),0\} ~ \mbox{a.e.} ~ t \in [0,T], ~ j \in J.
\end{align*}
\quad Define 
\begin{align*}
& \tilde{u}_{i}^{k+1}(t) = \mathrm{proj}_{[-M,M]}(u_{i}^{k}(t)) \in [-M,M] ~ \mbox{a.e.} ~ t \in [0,T], ~ i \in I, \\
& \tilde{v}_{j}^{k+1}(t) = \mathrm{proj}_{[0,N]}(v_{j}^{k}(t)) \in [0,N] ~ \mbox{a.e.} ~ t \in [0,T], ~ j \in J. \\
\end{align*}
\State \textbf{Step 5.} Set $k = k + 1$ and go to Step 1.
\Statex
\end{algorithmic}
\label{alg}
\end{algorithm}
 
The precise definition of an approximate solution in Step 2 of Algorithm \ref{alg} is as follows. There are multiple potential approaches for approximately solving the subproblems (P$_k$). For example, one may seek a global minimum, a local minimum, or a stationary point. Moreover, the convergence properties of the general method are subject to the approach employed in solving the subproblems. Accordingly, the concept of an approximate solution will be interpreted in accordance with the following hypothesis:
\begin{itemize}
\item[(H3)] There exists a sequence $\{\varepsilon^{k}\}_{k \in \mathbb{N}} \subset L^{1}([0, T]; \mathbb{R}^{n})$, with $\varepsilon^{k} \rightharpoonup 0$, %in $L^{1}([0,T];\mathbb{R}^{n})$ 
such that%, for each $(\rho_{k}, \tilde{u}^{k}, \tilde{v}^{k}) \in [0, +\infty) \times L^{\infty}([0,T];\mathbb{R}^{p}) \times L^{\infty}([0,T];\mathbb{R}_{+}^{m})$,
\begin{equation} \label{Eqr_1}
\nabla_{x}L_{\rho_{k}}(x^{k}(t), \tilde{u}^{k}(t), \tilde{v}^{k}(t),t) = \varepsilon^{k}(t) ~  \text{a.e.} ~ t \in [0,T] ~ \forall k \in \mathbb{N}.
\end{equation} 
\end{itemize}

The following theorem, under certain hypotheses, guarantees an optimality property for the limit points of the sequences generated by Algorithm \ref{alg}, when they exist.

%\begin{theorem} \label{Teo_Algo_Otima}
%Let $\bar{x}$ be a limit point of a sequence $\{x^{k}\}_{k\in \mathbb{N}}$ generated by Algorithm \ref{alg} in the sense that, under subsequences \textcolor{blue}{extraction} (we do not relabel), $x^{k}(t) \to \bar{x}(t)$ for almost every $t \in [0, T]$. If $\bar{x}$ is a feasible solution for (CTP) and Hypotheses (H1)--(H4) are satisfied, then $\bar{x}$ is a pw-AKKT solution.
%\end{theorem}
\begin{theorem} \label{Teo_Algo_Otima}
Let $\{x^{k}\}_{k\in \mathbb{N}}$ be a sequence generated by Algorithm \ref{alg} and $\bar{x} \in L^\infty([0,T];\mathbb{R}^n)$ satisfying $x^{k}(t) \to \bar{x}(t)$ for almost every $t \in [0, T]$. If $\bar{x}$ is a feasible solution for (CTP) and Hypotheses (H1)--(H3) are satisfied, then $\bar{x}$ is a pw-AKKT solution.
\end{theorem}
\begin{proof}
Let us take
\begin{align*}
& u_{i}^{k}(t) := \tilde{u}_{i}^{k}(t) + \rho_{k}h_{i}(x^{k}(t),t) ~  \text{a.e.} ~ t \in [0,T], ~ i \in I \\
& v_{j}^{k}(t) := \max\{\tilde{v}_{j}^{k}(t)+\rho_{k} g_{j}(x^{k}(t),t),0\} ~  \text{a.e.} ~ t \in [0,T], ~ j \in J,
\end{align*}
where sequences $\{\rho_k\}_{k \in \mathbb{N}}$, $\{\tilde{u}^k\}_{k \in \mathbb{N}}$ and $\{\tilde{v}^k\}_{k \in \mathbb{N}}$ are provided by Algorithm \ref{alg}. From (\ref{Eqr_1}), we obtain
\begin{align*}
\varepsilon^{k}(t) &= \nabla_x \phi(x^{k}(t),t) + \sum_{i=1}^{p}u_{i}^{k}(t)\nabla_x h_{i}(x^{k}(t),t) + \sum_{j=1}^{m}v_{j}^{k}(t)\nabla_x g_{j}(x^{k}(t),t) \\
&= \nabla_x L(x^{k}(t), u^{k}(t), v^{k}(t),t)  ~  \text{a.e.} ~ t \in [0,T], ~ \forall k \in \mathbb{N}.
\end{align*}
Since, by (H3), $\varepsilon^k \rightharpoonup 0$ in $L^1([0,T];\mathbb{R}^n)$, condition (\ref{2}) of Definition \ref{Def_Seq_AKKT} follows. Condition (\ref{4}) follows directly from the definition of the sequence $\{v^k\}_{k \in \mathbb{N}}$ above.

The verification of condition (\ref{3}) will be divided into two distinct cases. In the first case, we assume that $\rho_k \to \infty$. Let $t \in [0,T]$ and $j \in J$. Let us define 
$$
N_1 = \{ k : g_j(x^k(t),t) < 0 \} \quad \text{and} \quad N_2= \{ k : g_j(x^k(t),t) \geq 0 \}.
$$
For $k \in N_1$ large enough, taking into account that $\{\tilde{v}_{j}^{k}(t)\}$ is bounded by construction and $\{g_{j}(x^{k}(t),t)\}$ is bounded for it is convergent, we get 
\begin{equation} \label{cond1}
v_{j}^{k}(t) = \max\{\tilde{v}_{j}^{k}(t)+\rho_{k} g_{j}(x^{k}(t),t),0\} = 0.
\end{equation}
If no $k \in N_1$ is sufficiently large, it is because $N_1$ is finite and, under subsequence extraction, can be disregarded. For $k \in N_2$, we have directly
\begin{equation} \label{cond2}
g_{j}^{-}(x^{k}(t),t) = \max\{-g_{j}(x^{k}(t),t),0\} = 0.
\end{equation}
Now, we assume that sequence $\{\rho_k\}$ is bounded. Observing its construction in Algorithm \ref{alg}, we see that $\rho_k = \rho_{\bar{k}}$ for all $k \geq \bar{k}$. This means that
\begin{equation*}
\max\{\|H^k\|_{\infty},\|V^{k}\|_{\infty}\} \leq \tau^{k-\bar{k}} \max\{\|H^{\bar{k}}\|_{\infty},\|V^{\bar{k}}\|_{\infty}\} ~ \forall k \geq \bar{k}.
\end{equation*}
Hence, $\|V^{k}\|_{\infty} \to 0$, since $\tau \in (0,1)$. In other words,
\begin{equation} \label{convVk}
V_{j}^{k}(t) = \max\{g_{j}(x^{k}(t),t),-\tilde{v}_{j}^{k}(t)/\rho_{k}\} \to 0 ~ \mbox{a.e.} ~ t \in [0,T], ~ j \in J.
\end{equation}
Let $t \in [0,T]$ and $j \in J$. Let us define 
$$
N_3 = \{ k : g_{j}(x^{k}(t),t) \geq -\tilde{v}_{j}^{k}(t)/\rho_{k} \} \quad \text{and} \quad N_4= \{ k : g_{j}(x^{k}(t),t) < -\tilde{v}_{j}^{k}(t)/\rho_{k} \}.
$$
For $k \in N_3$, we have $\tilde{v}_{j}^{k}(t) + \rho_{k} g_{j}(x^{k}(t),t) \geq 0$, so that $v_{j}^{k}(t) = \max\{\tilde{v}_{j}^{k}(t)+\rho_{k} g_{j}(x^{k}(t),t),0\} = \tilde{v}_{j}^{k}(t) + \rho_{k} g_{j}(x^{k}(t),t)$. Taking into account that $\{\tilde{v}_{j}^{k}(t)\}$ is bounded by construction, $\{\rho_k\}$ is bounded by assumption and $\{g_{j}(x^{k}(t),t)\}$ is bounded for it is convergent, we see that $\{v_j^k(t)\}$ is bounded. On the other hand, $V_j^k(t) = g_j(x^k(t),t)$, so that, by (\ref{convVk}), $g_j(x^k(t),t) \to 0$ as $k \to \infty$ over $N_3$. If $N_3$ is finite, as before, it can be disregarded. Putting these together gives us
\begin{equation} \label{cond3}
v_j^k(t) g_{j}^{-}(x^{k}(t),t)  = v_j^k(t) \max\{-g_{j}(x^{k}(t),t),0\} \to 0. 
\end{equation}
Finally, for $k \in N_4$, we directly have
\begin{equation} \label{cond4}
v_{j}^{k}(t) = \max\{\tilde{v}_{j}^{k}(t)+\rho_{k} g_{j}(x^{k}(t),t),0\} = 0.
\end{equation}
By (\ref{cond1}), (\ref{cond2}), (\ref{cond3}) and (\ref{cond4}) we obtain
\begin{equation*}
v_j^k(t) g_{j}^{-}(x^{k}(t),t)  \to 0 ~ \text{a.e.} ~ t \in [0,T], ~ j \in J,
\end{equation*}
that is, condition (\ref{3}) of Definition \ref{Def_Seq_AKKT} is satisfied. Thus, $\{(x^k,u^k,v^k)\}$ is an AKKT sequence. As $x^k(t) \to \bar{x}(t)$ for almost every $t \in [0,T]$ by assumption, we see that $\bar{x}$ is a pw-AKKT solution of (CTP). 
\end{proof}
\medskip

In Theorem \ref{Teo_Algo_Otima}, the existence of $\bar{x} \in L^\infty([0,T];\mathbb{R}^n)$ satisfying $x^{k}(t) \to \bar{x}(t)$ for almost every $t \in [0,T]$ is guaranteed, for example, if $\bar{x}$ is a limit point of the sequence $\{x^{k}\}_{k\in \mathbb{N}}$ in the $L^1$ norm (as in the proof of Theorem \ref{Teo_AKKT}). Moreover, the feasibility of the point $\bar{x}$ is assumed. It must be acknowledged that it is not possible to guarantee the validity of this assumption; however, it can be asserted that $\bar{x}$ represents a robust candidate for feasibility, as demonstrated in Theorem \ref{Teo_Viabi_Algo}, presented immediately below.

\begin{theorem} \label{Teo_Viabi_Algo}
%Let $\bar{x}$ be a limit point of a sequence $\{x^{k}\}_{k\in \mathbb{N}}$ generated by Algorithm \ref{alg} in the sense that, under subsequences \textcolor{blue}{extraction} (we do not relabel), 
Let $\{x^{k}\}_{k\in \mathbb{N}}$ be a sequence generated by Algorithm \ref{alg} and $\bar{x} \in L^\infty([0,T];\mathbb{R}^n)$ satisfying $x^{k}(t) \to \bar{x}(t)$ for almost every $t \in [0, T]$. Assume that (H1) -- (H3) are satisfied. In addition, assume that
\begin{itemize}
\item[(H4)] There exists an integrable  function $k_{\phi,h,g}$ such that 
$$
\Vert \nabla_x \phi(x,t) \Vert + \Vert \nabla_x h(x,t) \Vert + \Vert \nabla_x g(x,t) \Vert \leq k_{\phi,h,g}(t) ~ \forall \, x \in \bar{B}_{\delta}(t) ~ \text{a.e.} ~ t \in [0,T]
$$
for some $\delta >0$.
\end{itemize}
%If the limit point $\bar{x}$ belongs to $L^\infty([0,T];\mathbb{R}^n)$, 
Then $\bar{x}$ is a stationary point for the problem of minimizing the feasibility factor function:
\begin{equation} \label{feas_fac_func}
\begin{array}{ll}
\text{minimize} & \Theta(x) = \displaystyle{\int_{0}^{T} \left[ \sum_{i=1}^{p} \Big[ h_{i}(x(t),t) \Big]^2 + \sum_{j=1}^{m} \left[ g^{+}_{j}(x(t),t) \right]^2 \right] ~ dt} \\
\text{subject to} & x \in L^\infty([0,T];\mathbb{R}^n).
\end{array}
\end{equation}
\end{theorem}
\begin{proof}
If $\{\rho_{k}\}$ has a bounded subsequence, reasoning as in the proof of Theorem \ref{Teo_Algo_Otima}, we know from (\ref{convVk}) that
\begin{equation*}
V_{j}^{k}(t) \to 0 ~\text{a.e.} ~ t \in [0,T], ~ j \in J.
\end{equation*}
From the definition of $V_{j}^{k}(t)$, we get
$$
g_j(x^k(t),t) \leq V_j^k(t) ~ \mbox{a.e.} ~ t \in [0,T], ~ j \in J.
$$
Taking limits, we obtain $g_j(\bar{x}(t),t) \leq 0$ for almost every $t \in [0,T]$, $j \in J$. Similarly, 
$$
H_i^{k}(t) \rightarrow 0 ~\text{a.e.} ~ t \in [0,T], ~ i \in I,
$$ 
implying that $h_{i}(\bar{x}(t),t) = 0 $ for almost every $t \in [0,T]$, $i \in I.$ Then, $\bar{x}$ is an optimal solution for the Problem (\ref{feas_fac_func}). According to \cite[Prop. 3.2]{monte:2019}, this implies that $\bar{x}$ is a stationary point.
    
We will now analyze the case where $\{\rho_{k}\}_{k\in \mathbb{N}}$ does not have a bounded subsequence, that is $\rho_k \to \infty$. By (H3), we know that there exists $\{\varepsilon^{k}\}_{k \in \mathbb{N}} \subset L^{1}([0,T];\mathbb{R}^{n})$ with $\varepsilon^{k} \rightharpoonup 0$ such that (\ref{Eqr_1}) holds. By integrating both sides of (\ref{Eqr_1}) over $[0,T]$ and dividing by $\rho_{k}$, we have
\begin{align*}
& \dfrac{1}{\rho_{k}} \int_{0}^{T} \varepsilon^{k}(t) \cdot \gamma(t) ~dt \\
& \quad = \dfrac{1}{\rho_{k}} \int_{0}^{T} \nabla_x \phi(x^{k}(t),t) \cdot \gamma(t) ~dt \\
& \qquad + \int_{0}^{T} \sum_{i=1}^{p} \left[ \dfrac{\tilde{u}_{i}^{k}(t)}{\rho_{k}} + h_{i}(x^{k}(t),t) \right] \nabla_x h_{i}(x^{k}(t),t) \cdot \gamma(t) ~dt \\
& \qquad + \int_{0}^{T} \sum_{j=1}^{m} \left[ \max \left\{ 0, \dfrac{\tilde{v}_{j}^{k}(t)}{\rho_{k}} + g_{j}(x^{k}(t),t) \right\} \right] \nabla_x g_{j}(x^{k}(t),t) \cdot \gamma(t) ~dt
\end{align*}
for all $\gamma \in L^{\infty}([0, T]; \mathbb{R}^{n})$ and each $k\in \mathbb{N}$. Therefore, taking the limit as $k \to \infty$, by the weak convergence of $\{\varepsilon^{k}\}$, the pointwise convergence of $\{x^{k}\}$, Hypothesis (H4), the continuous differentiability of $\phi, h$ and $g$, the boundedness of sequences $\{\tilde{u}^{k}\}$ and $\{\tilde{v}^{k}\}$, and the Dominated Convergence Theorem (see, for example, \cite[Thm. 3.25]{Gordon1994}), we conclude that
\begin{multline*}
\int_0^T \left[ \sum_{i=1}^{p}h_{i}(\bar{x}(t),t)\nabla_x h_{i}(\bar{x}(t),t) \right. \\ + \int_0^T \left. \sum_{j=1}^{m}\max\left\{0, g_{j}(\bar{x}(t),t)\right\}\nabla_x g_{j}(\bar{x}(t),t) \cdot \gamma(t) \right] ~dt = 0
\end{multline*}
for all $\gamma \in L^\infty([0,T];\mathbb{R}^n)$, which, in turn, means that $\bar{x}$ is a stationary point for (\ref{feas_fac_func}). 
\end{proof}

\section{Applications} \label{appl}

This section presents some applications of the augmented Lagrangian method to problems of the form (CTP). The implementation was performed using \textsc{Matlab}\textsuperscript{\circledR} \cite{MATLAB2022}. We set $\rho_{0} = 1$, $\gamma = 1.001$, and $\tau = 10^{-3}$. The general stopping criterion for Algorithm \ref{alg} was based on conditions \eqref{2} and \eqref{3} from Definition \ref{Def_Seq_AKKT}. A precision of $10^{-5}$ was employed, the maximum number of iterations was set to $10^{3}$, and the interval $[0,T]$ was uniformly discretized using $85$ points. To illustrate the practical feasibility of the proposed algorithm, the numerical results of each of the following four examples are presented.

\begin{example}[do Monte {\cite[Example 4.1]{doMonte2018}}] \label{Ex_1}
Let us consider the following problem:
$$
\begin{array}{ll}
\text{minimize} & \displaystyle{\int_{0}^{1} \left[ x_{1}(t)^{2} + x_{2}(t) \right] ~dt} \\
\text{subject to} & -x_{2}(t) \leq 0 ~ \text{a.e.} ~ t \in [0,1], \\
& -x_{1}(t)^{2} - x_{2}(t) \leq 0 ~ \text{a.e.} ~ t \in [0, 1], \\
& x \in L^{\infty}([0,1];\mathbb{R}^{2}).
\end{array}
$$
An optimal solution for this problem is given by $\bar{x}(t) \equiv (0,0)$. Although the full rank \cite{monte:2019} and constant rank constraint \cite{monte:2021} qualifications are not satisfied at $\bar{x}$, the proposed version of the augmented Lagrangian method effectively solved the problem. The results are shown in Figure \ref{Fig_Ex_1}.
\begin{figure}[htbp]
\centering
\includegraphics[width=10.0cm]{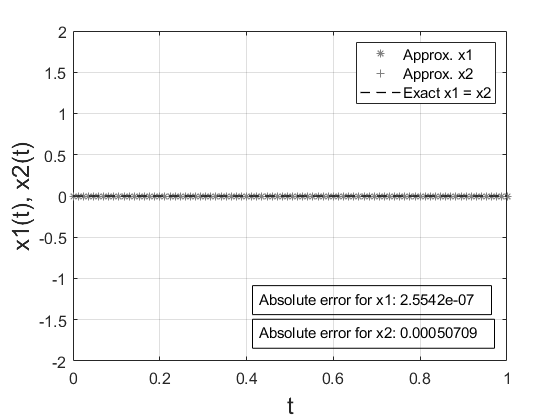}
%\vspace{-0.2 cm}
\caption{Numerical results for $(x_{1}^{0}(t), x_{2}^{0}(t)) \equiv (1, 1)$, $N = 10^{50}$, and $(\tilde{v}_{1}^{1}(t), \tilde{v}_{2}^{1}(t)) \equiv (1, 1)$. The method performed 2 iterations.}
\label{Fig_Ex_1}
\end{figure}
\end{example}

\begin{example}[do Monte and de Oliveira {\cite[Example 3.6]{monte:2020}}] \label{Ex_2}
Consider the continuous-time programming problem presented below:
$$
\begin{array}{ll}
\text{minimize} & \displaystyle{\int_{0}^{1} x_{1}(t) ~dt} \\
\text{subject to} & x_{1}(t)^{2} - 2x_{1}(t) + x_{2}(t) - t \leq 0 ~ \text{a.e.} ~ t \in [0,1], \\
& x_{1}(t)^{2} - 2x_{1}(t) - x_{2}(t) + t \leq 0 ~ \text{a.e.} ~ t \in [0,1], \\
& -x_{1}(t)^{2} + \dfrac{1}{2}x_{1}(t) + x_{2}(t) - t \leq 0 ~ \text{a.e.} ~ t \in [0,1], \\
& x \in L^{\infty}([0,1];\mathbb{R}^{2}).
\end{array}
$$
An optimal solution for this problem is $\bar{x}(t) = (0,t)$ a.e. on $[0,1]$. The results obtained by Algorithm \ref{alg} can be found in Figure \ref{Fig_Ex_2}.
\begin{figure}[htbp]
\centering
\includegraphics[width=10.0cm]{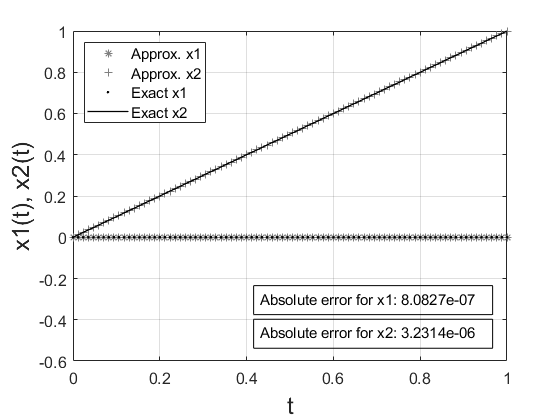}
%\vspace{-0.2 cm}
\caption{Numerical results for $(x_{1}^{0}(t), x_{2}^{0}(t)) \equiv (0.5, 0.5)$, $N = 10^{50}$, and $(\tilde{v}_{1}^{1}(t),\tilde{v}_{2}^{1}(t),\tilde{v}_{3}^{1}(t)) \equiv (1,1,1)$, after 78 iterations.}
\label{Fig_Ex_2}
\end{figure}
\end{example}

\begin{example}[do Monte {\cite[Example 4.4]{doMonte2018}}] \label{Ex_3}
Consider the following problem:
$$
\begin{array}{ll}
\text{minimize} &  \displaystyle{\int_{0}^{1} \left[ (x_{1}(t) - 1)^{2} + (x_{2}(t) - 1)^{2} - x_{3}(t)^{2} \right] ~dt} \\
\text{subject to} & x_{1}(t)^{2} + x_{2}(t)^{2} - x_{3}(t) - 2 = 0 ~ \text{a.e.} ~ t \in [0,1], \\
& 2x_{1}(t)x_{2}(t) - 4x_{2}(t) - x_{3}(t) + 2 \leq 0 ~ \text{a.e.} ~ t \in [0,1], \\
& -x_{1}(t) - \dfrac{1}{2}x_{3}(t) + 1 \leq 0 ~ \text{a.e.} ~ t \in [0,1], \\
& x \in L^{\infty}([0,1];\mathbb{R}^{3}).
\end{array}
$$
The feasible point $\bar{x} \equiv (1,1,0)$ is an optimal solution to this problem. The Algorithm \ref{alg} produced the results shown in Figure \ref{Fig_Ex_3}.
\begin{figure}[htbp]
\centering
\includegraphics[width=10.0cm]{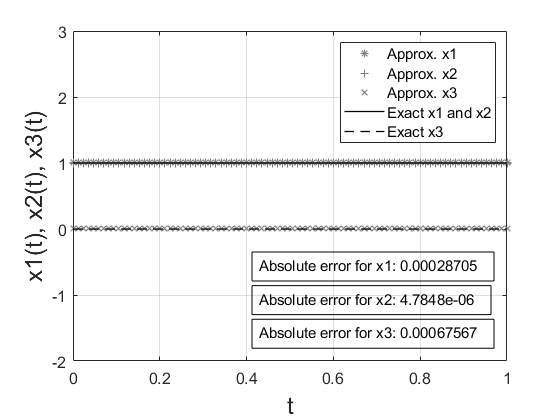}
%\vspace{-0.2 cm}
\caption{Numerical results for $(x_{1}^{0}(t),x_{2}^{0}(t),x_{3}^{0}(t)) \equiv (-100,-100,-100)$, $M = N = 10^{50}$, $(\tilde{u}_{1}^{1}(t),\tilde{v}_{1}^{1}(t),\tilde{v}_{2}^{1}(t)) \equiv (1,1,1)$, and 20 iterations.}
\label{Fig_Ex_3}
\end{figure}
\end{example}

\begin{example}[de Oliveira {\cite[Example 1]{deOliveira2024}}] \label{Ex_4}
Consider the problem given as
$$
\begin{array}{ll}
\text{minimize} & \displaystyle{\int_{0}^{2} c(t)^{T}x(t) ~dt} \\
\text{subject to} & A(t)x(t) - b(t) \leq 0 ~ \text{a.e.} ~ t \in [0, 2], \\
& x \in L^{\infty}([0,2];\mathbb{R}^{2}),
\end{array}
$$
where, for almost every $t \in [0,2]$,
\begin{align*}
c(t) = \left[ \begin{array}{c} (t - 1)\sgn(1 - t) \\ -1 \end{array} \right],
\end{align*}
\begin{align*}
A(t) = \left[ \begin{array}{cc} 0 & -1 \\ -1 & 0 \\ \sgn(t - 1) & \sgn(1 - t) \\ 1 & 1 \\ 0 & 1 \end{array} \right], ~
\end{align*}
and
\begin{align*}
b(t) = \left[ \begin{array}{c} 0 \\ 0 \\ 0 \\ 3 \\ \dfrac{1}{4} + \dfrac{5}{8}t \end{array} \right].
\end{align*}
In \cite{deOliveira2024}, de Oliveira shows that an optimal solution to this problem is given by
$$
\bar{x}_{1}(t) = 
\begin{cases} 
\dfrac{11}{4} - \dfrac{5}{8}t ~ \text{a.e.} ~ t \in [0,1], \\ \dfrac{1}{4} + \dfrac{5}{8}t ~ \text{a.e.} ~ t \in (1,2], 
\end{cases}
\quad \text{and} \quad
\bar{x}_{2}(t) = \dfrac{1}{4} + \dfrac{5}{8}t ~ \text{a.e.} ~ t \in [0,2].
$$
Figure \ref{Fig_Ex_4} shows the numerical results obtained by the proposed augmented Lagrangian method and the optimal solution.
\begin{figure}[htbp]
\centering
\includegraphics[width=10.0cm]{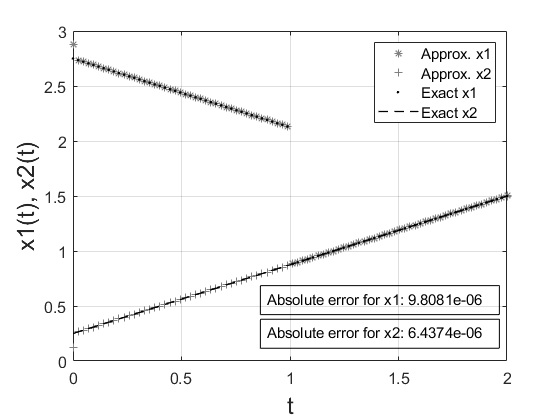}
%\vspace{-0.2 cm}
\caption{Numerical results for $(x_{1}^{0}(t),x_{2}^{0}(t)) \equiv (1,1)$, $N = 10^{50}$, and $(\tilde{v}_{1}^{1}(t),\tilde{v}_{2}^{1}(t),\tilde{v}_{3}^{1}(t),\tilde{v}_{4}^{1}(t)) \equiv (1,1,1,1)$. The method performed 98 iterations.}
\label{Fig_Ex_4}
\end{figure}
\end{example}

In Examples \ref{Ex_1}, \ref{Ex_2}, \ref{Ex_3}, and \ref{Ex_4}, a range of scenarios are examined that present a variety of challenges for Algorithm \ref{alg}. In Example \ref{Ex_1}, we deal with a nonlinear problem in which the full-rank and constant-rank CQs are not satisfied. Example \ref{Ex_2} addresses a problem with nonlinear constraints that explicitly depend on the variable $t$. In Example \ref{Ex_3}, we analyse a nonlinear problem with both equality and inequality constraints. Finally, Example \ref{Ex_4} presents a linear problem with a discontinuous solution. The outcomes illustrated in Figures \ref{Fig_Ex_1}, \ref{Fig_Ex_2}, \ref{Fig_Ex_3}, and \ref{Fig_Ex_4} illustrate that the proposed version of the augmented Lagrangian method effectively addressed each of the aforementioned challenges, including the capture of the discontinuity present in the solution of the problem in Example \ref{Ex_4}.

\section{Conclusion} \label{sec:Conclusion}

In this paper, we present a novel approach to continuous-time optimization problems by defining asymptotic-type optimality conditions, which we refer to as pw-AKKT conditions. We demonstrate that every local optimal solution satisfies such conditions, thereby establishing them as genuine necessary optimality conditions. No constraint qualification was imposed in the problem data. This kind of optimality conditions is particularly useful in a practical sense, as they can be used as precise stopping criteria in numerical methods.

In this regard, a significant contribution of this work is the adaptation of the augmented Lagrangian method with pw-AKKT conditions as a general stopping criterion. The pw-AKKT condition provides a versatile termination criterion that addresses a crucial yet frequently overlooked aspect in the existing literature on the subject.

In our upcoming research, we intend to apply this theoretical framework to real-world physical problems modelled as continuous-time problems (CTP) and assess its effectiveness in fields such as control theory and economics, where these models are frequently utilized.

\section*{Acknowledgment}

This research was supported by [APQ-00453-21, Minas Gerais Research Foundation (FAPEMIG)], [2022/16005-0, S\~ao Paulo Research Foundation (FAPESP)] and [305245/2024-4, National Council for Scientific and Technological Development (CNPq)].

%\printbibliography
\bibliographystyle{plain}
\bibliography{AKKTContinuousTime}

%%%%%%%%%%%%%%%%%
\end{document}